\documentclass{article}
\usepackage[utf8]{inputenc}
\usepackage{amsmath,amssymb,amsfonts,amsthm,enumitem,xy,pstricks,pst-node,graphicx,comment}
\xyoption{all}

\title{Smooth Approximations of Quasispheres}
\author{Spencer Cattalani\thanks{Partially supported by NSF grant DMS 2246485 and the Simons Foundation}}
\date{\today}

\numberwithin{equation}{section}
\newtheorem{thm}{Theorem}
\newtheorem{prp}[thm]{Proposition}
\newtheorem{lmm}[thm]{Lemma}   
\newtheorem{crl}[thm]{Corollary}

\theoremstyle{definition}

\theoremstyle{remark}
\newtheorem{rmk}[thm]{Remark}

\theoremstyle{remark}

\def\BE#1{\begin{equation}\label{#1}}
\def\EE{\end{equation}}
\def\e_ref#1{(\ref{#1})}

\def\wt#1{\widetilde{#1}}

\def\lra{\longrightarrow}

\def\R{\mathbb R}

\def\N{\mathbb N}

\def\st{\mathrm{St}}

\def\distortion{\lambda_\varepsilon}
\def\continuousmetric{d_\varepsilon}
\def\smoothmetric{\rho_\varepsilon}

\begin{document}

\maketitle

\begin{abstract}
We prove that every $n$-dimensional quasisphere is the Gromov-Hausdorff limit of a sequence of locally smooth uniform quasispheres, for every dimension $n$. We also prove an analogous result in the bi-Lipschitz setting. This extends recent results of D.~Ntalampekos from dimension 2 to arbitrary dimension. In the process, we replace the second half of his argument by a completely different, more efficient approach, which should be applicable to other problems.
\end{abstract}

\section{Introduction}
Uniformization is a classical theme of conformal geometry, going back to the uniformization of Riemann surfaces. In the context of metric spaces, \textit{quasisymmetric} parameterizations are often considered. This class of mappings has a rich history and connections to fields such as dynamics and geometric topology; we refer the reader to \cite{N25b} for an account of these developments. Special interest is often given to metric spaces which are quasisymmetrically parameterized by the sphere, so called \textit{quasispheres}.\\

A related topic, which has had applications to uniformization \cite{NR24,NR23}, is that of \textit{approximation}. Often, one is concerned with approximating a metric space in the Gromov-Hausdorff sense by Riemannian manifolds. This involves a delicate interplay; one would like to approximate by a nice class of spaces, but such classes are frequently closed under Gromov-Hausdorff convergence. For example, a geodesic metric on a smooth manifold can be approximated by Riemannian metrics on that same manifold \cite{FO95} and, in fact, every geodesic metric space can be approximated by Riemannian surfaces \cite{C92}. However, this would be impossible if one were to additionally require control over the topology of the surfaces \cite{D24}. A similar rigidity with respect to Ricci curvature is a celebrated result in differential geometry \cite{CC97}.\\

A fascinating recent development in this program is Ntalampekos's characterization of quasispheres in terms of smooth approximations \cite[Theorem~1.8]{N25}. In particular, he shows that every quasisphere is the Gromov-Hausdorff limit of a sequence of Riemannian quasispheres with uniform bounds. This result is notable for requiring no additional assumptions, such as Ahlfors regularity or finiteness of area. It therefore provides a general means of understanding quasispheres. Our main result is a generalization of this this theorem to arbitrary dimensions.

\begin{thm}\label{quasi_thm}
Let $(X,d)$ be a compact connected metric space. If $(X,d)$ is quasisymmetric to a Riemannian manifold $(X,d_g)$, then $(X,d)$ is the Gromov-Hausdorff limit of a sequence of metric spaces which are locally isometric to Riemannian manifolds and uniformly quasisymmetric to $(X,d_g)$.
\end{thm}

The constants implicit in the above conclusion depend only on the quasisymmetric distortion and the dimension of $X$. The converse also holds; the proof of this direction of \cite[Theorem~1.8]{N25} works in any dimension. We note that, in case $(X,d)$ is a length space, the approximating sequence can be taken to be Riemannian, not just locally so. This is explained in Remark~\ref{remove_locally_rmk}. Remark~\ref{lack_of_compactness_rmk} discusses the degree to which the compactness assumption is necessary.\\

The proof of Theorem~\ref{quasi_thm} follows the same general scheme as that in \cite{N25} and, in particular, uses a special case of its main technical result \cite[Theorem~2.8]{N25} (stated in the present note as Proposition~\ref{main_technical_prp}). The difference in the proofs lies in the construction to which Proposition~\ref{main_technical_prp} is applied. In \cite[Sections~3 and 4]{N25}, the Riemannian manifold $(X,d_g)$ is first triangulated in a controlled way, then the triangulation is modified and finally re-smoothed. The modification step of \cite{N25} only works in dimension 2 and the re-smoothing step only works in dimension up to 4. The construction in Section~\ref{main_sec} of the present note instead remains within the smooth setting. We conformally rescale the metric by a function $\distortion$, which measures the length distortion at scale $\varepsilon$ (cf. \cite[Section~7.8]{HK98}). This construction not only works in any dimension, but is much more simply described. We hope that the simplicity of the construction will be of further use even in the 2-dimensional case.\\

A similar result also holds in the bi-Lipschitz case, generalizing \cite[Theorem~1.11]{N25} to arbitrary dimension:

\begin{thm}\label{Lipschitz_thm}
Let $(X,d)$ be a compact connected metric space. If $(X,d)$ is $L$-bi-Lipschitz homeomorphic to a Riemannian manifold $(X,d_g)$, then $(X,d)$ is the Gromov-Hausdorff limit of a sequence of metric spaces which are locally isometric to $(X,Ld_g)$ and $L$-bi-Lipschitz homeomorphic to $(X,d_g)$.
\end{thm}

Compared to \cite[Theorem~1.11]{N25}, Theorem~\ref{Lipschitz_thm} not only holds in arbitrary dimension, but also retains complete control over the Lipschitz constant. The proof of Theorem~\ref{Lipschitz_thm} follows the proof in \cite{N25}. In this case, the simplification is even more drastic, as the gluing construction of \cite[Section~2.2]{N25} and a uniform scaling are all that are required.\\

The structure of this note is as follows. Sections~\ref{prelim_sec} and \ref{approx_sec} recall the necessary background from \cite{N25}. The focus of Section~\ref{prelim_sec} is Lemma~\ref{gluing_lmm}. This alone suffices to prove Theorem~\ref{Lipschitz_thm}. The focus of Section~\ref{approx_sec} is Proposition~\ref{main_technical_prp}, which we show to be a very special case of \cite[Theorem~2.8]{N25}. It is a local-to-global result for quasisymmetries which serves as the technical heart of the proof of \cite[Theorem~1.8]{N25} and Theorem~\ref{quasi_thm}. Instead of the $(K,L)$-approximations of~\cite{N25}, we restrict Proposition~\ref{main_technical_prp} to $\varepsilon$-nets; Lemma~\ref{net_to_approx_lmm} justifies this simplification. Section~\ref{main_sec} consists of the novel construction of this note and contains the proofs of Theorems~\ref{quasi_thm} and \ref{Lipschitz_thm}, replacing the twelve pages of \cite[Sections~3 and 4]{N25}.\\

\textbf{Acknowledgements:} The author would like to thank Dimitrios Ntalampekos for his guidance during the writing of this note and Aleksey Zinger and the anonymous referee for helpful comments which greatly improved the exposition.

\section{Basic metric and Riemannian geometry}\label{prelim_sec}
Let $d$ and $\rho$ be metrics on a space $X$. Let $\eta: [0,\infty) \lra {[}0,\infty)$ be a homeomorphism. We say the metric $d$ is \textit{$\eta$-quasisymmetric} to $\rho$ if the identity map $\textnormal{id}: (X,d) \lra (X,\rho)$ is an $\eta$-quasisymmetry, i.e.
\begin{equation}\label{quasisymmetry_inequality}
\rho(p,q) \leq \eta \bigg(\frac{d(p,q)}{d(q,s)}\bigg)\rho(q,s)
\end{equation}
for all points $p$, $q$, and $s \neq q$ in $X$. Let $L \geq 1$. The metrics $d$ and $\rho$ are \textit{$L$-bi-Lipschitz} if
$$L^{-1}d(p,q) \leq \rho(p,q) \leq L d(p,q)$$
for all points $p$ and $q$ in $X$. We denote the open $r$-ball centered at $p \in X$ by $B(p,r)$ and the closed ball by $\overline{B}(p,r)$. We note that the closed ball is the set of points $q \in X$ such that $d(p,q) \leq r$, not necessarily the closure of the open ball. If $A \subset X$, then $N(A,r)$ denotes the open $r$-neighborhood of the set $A$. The \textit{Hausdorff distance} between subsets $A$ and $B$ of a metric space is the infimal $r \geq 0$ such that $B \subset N(A,r)$ and $A \subset N(B,r)$. The \textit{Gromov-Hausdorff distance} between two metric spaces $X$ and $Y$ is the infimal Hausdorff distance between $f(X)$ and $g(Y)$ over all metric spaces $Z$ and isometric embeddings $f: X \lra Z$ and $g: Y \lra Z$.

\begin{lmm}[{\cite[p78]{H01}}]\label{bi-Lip_are_QS_lmm}
Let $d,\rho$ be metrics on the same space. If $d$ is $L$-bi-Lipschitz to $\rho$, then $d$ is $\eta$-quasisymmetric to $\rho$, where $\eta(t) = L^2t$.
\end{lmm}

\begin{lmm}[{\cite[Corollary~11.5]{H01}}]\label{Holder_lmm}
If $(X,d_g)$ is a compact Riemannian manifold and $d_g$ is quasisymmetric to a metric $d$, then there exist $C \geq 1$ and $\alpha \leq 1$ such that
$$d(p,q) \leq C d_g(p,q)^\alpha$$
for $p,q \in X$ with $d_g(p,q) < 1$.
\end{lmm}

\begin{lmm}[{\cite[Proposition~95]{B03}}]\label{local_control_lmm}
For each compact Riemannian $n$-manifold $(X,d_g)$, there exists $c>0$ such that for each $r \in (0,c)$ and point $p \in X$, $B(p, r)$ is convex and $2$-bi-Lipschitz homeomorphic to an $\varepsilon$-ball in~$\R^n$ with the standard Euclidean metric $d_{std}$.
\end{lmm}

\begin{crl}\label{throw_crl}
Let $(X,d_g)$ be a compact Riemannian manifold. For every $r > 0$ sufficiently small, and $p,q \in X$ such that $d_g(p,q) < r$, there exists a point $s \in X$ such that
\begin{equation}\label{throw_ineq_1}
8^{-1}r \leq d_g(p,s) \leq 8r \quad \textnormal{and} \quad 8^{-1}r \leq d_g(q,s) \leq 8r.
\end{equation}
\end{crl}

\begin{proof}
By Lemma~\ref{local_control_lmm}, we can associate $p$ and $q$ with points in $\R^n$ such that $d_{std}(p,q) < 2r$. Therefore, there is a point $s$ on the line through $p$ and $q$ such that 
$$d_{std}(p,s) = 4r = d_{std}(p,q) + d_{std}(q,s).$$
It follows from the fact that
$$d_{std}(q,s) = 4r - d_{std}(p,q) \quad  \textnormal{and}\quad 2r \leq 4r - d_{std}(p,q) \leq 4r$$
that
$$ 2r \leq d_{std}(q,s) \leq 4r.$$
As $d_{std}$ is $2$-bi-Lipschitz to $d_g$,
\begin{equation*}
\begin{split}
2r &= 2^{-1}d_{std}(p,s) \leq d_g(p,s) \leq 2d_{std}(p,s) = 8r \quad \textnormal{and}\\
r &\leq 2^{-1}d_{std}(q,s) \leq d_g(q,s) \leq 2d_{std}(q,s) \leq 8r.
\end{split}
\end{equation*}
Thus, $s\in X$ satisfies (\ref{throw_ineq_1}).
\end{proof}

\begin{lmm}[{\cite[Theorem~52]{B03}}]\label{Hopf_Rinow_lmm}
The distance between any two points of a compact Riemannian manifold is realized as the length of a length-minimzing geodesic connecting them.
\end{lmm}

\begin{lmm}[{\cite[Lemma~2.2]{N25}}]\label{gluing_lmm}
Let $(X,\rho)$ be a metric space and $S\subset X$ be a closed set with a metric $d$ such that $d\leq \rho$ on $S\times S$. Then, the function
\begin{align*}
\wt{\rho}: X \times X \lra {[}0,\infty {)}, \quad \wt{\rho}(x,y) := \min \bigl\{\rho(x,y), \inf_{p,q\in S} \{\rho(x,p)+d(p,q)+\rho(q,y)\}\bigr\},
\end{align*}
is a metric on $X$ such that $\wt{\rho} \leq \rho$. If $(S,d)$ is a discrete metric space, then the identity map from $(X,\rho)$ onto $(X,\wt{\rho})$ is a local isometry. If $Ld \geq \rho$ on $S \times S$ for some $L \geq 1$, then $\wt{\rho}$ is $L$-bi-Lipschitz to $\rho$.
\end{lmm}
We say that $\wt{\rho}$ is the \textit{glued metric determined by $\rho$ and $d$}. 

\section{Nets and approximations}\label{approx_sec}
Let $(X,d)$ be a metric space and $\varepsilon > 0$. A subset $S \subset X$ is called \textit{$\varepsilon$-dense} if, for each point $p \in X$, there is a point $s \in S$ such that $d(p,s) < \varepsilon$. A subset $S \subset X$ is called \textit{$\varepsilon$-separated} if $d(s,s') \geq \varepsilon$ for all pairs of distinct points $s$ and $s'$ in $S$. An \textit{$\varepsilon$-net} is an $\varepsilon$-dense and $\varepsilon$-separated subset. The following is an easy consequence of Zorn's lemma.

\begin{lmm}[{\cite[Exercise~12.10]{H01}}]\label{net_lmm}
Every metric space contains an $\varepsilon$-net for every $\varepsilon > 0$.
\end{lmm}

\begin{prp}[{special case of \cite[Theorem~2.8]{N25}}]\label{main_technical_prp}
Let $n \in \N$, $R \geq 1$ be sufficiently large (depending on $n$), $\varepsilon > 0$ be sufficiently small (depending on $R$), $X$ be a compact connected smooth $n$-manifold with metrics $d_g$, $\rho$, and $d$, and $S \subset (X,d_g)$ be an $\varepsilon$-net. Suppose $d_g$ and $\rho$ are induced by Riemannian metrics on $X$, $d \leq \rho$ on $S \times S$, and there exist $L \geq 1$ and a homeomorphism $\eta_1: {[}0,\infty{)} \lra {[}0,\infty{)}$ such that
\begin{equation}\label{main_technical_ineq}
d(s,s') \geq L^{-1} \rho(s,s') \quad \forall \, s, s' \in S \quad  \textnormal{such that} \quad d_g(s,s') < 2\varepsilon,
\end{equation}
$d_g$ is $\eta_1$-quasisymmetric to $d$, and $d_g$ is $\eta_1$-quasisymmetric to $\rho$ on $B(s,R\varepsilon)$ for every $s \in S$.
Then, there exists a homeomorphism $\eta_2: {[}0,\infty{)} \lra {[}0,\infty{)}$, depending only on $R$, $L$, and $\eta_1$ such that $d_g$ is $\eta_2$-quasisymmetric to the glued metric $\wt{\rho}$ of Lemma~\ref{gluing_lmm} determined by $\rho$ and $d|_{S \times S}$.
\end{prp}

The statement of \cite[Theorem~2.8]{N25} is much more general, as it applies to maps between metric spaces and with $(K,L)$-approximations instead of $\varepsilon$-nets. In the remainder of this section, we show that Proposition~\ref{main_technical_prp} is indeed a special case of \cite[Theorem~2.8]{N25}. Let $(X,d_g)$ be a Riemannian manifold.
For the ease of referencing and with the author's permission, the next paragraph is taken almost verbatim from \cite[Section~2.3]{N25}.\\

Given a graph $G=(V,\sim)$, we denote by $k(u,v)$ the combinatorial distance between vertices $u,v\in V$, i.e. the minimum number of edges in a chain connecting the two vertices. Note that $k(u,v)$ is understood to be $\infty$ if there is no chain of edges connecting $u$ and $v$. We consider quadruples $\mathcal A=(G,\mathfrak{p},\mathfrak{r}, \mathcal{U})$, where $G=(V,\sim)$ is a graph with vertex set $V$, $\mathfrak{p}: V\lra X$ and $\mathfrak{r}: V\lra (0,\infty)$ are maps, and $\mathcal U=\{\mathcal U(v): v\in V\}$ is an open cover of~$X$. We~let
$$p_v := \mathfrak{p}(v), \quad r_v := \mathfrak{r}(v), \quad \textnormal{and} \quad U_v := \mathcal U(v) \quad \forall\, v\in V.$$
For $K>0$, we define the \textit{$K$-star} of a vertex $v\in V$ with respect to $\mathcal A$ as
$$\mathcal A\text{-}\st_K(v) := \!\!\! \bigcup_{\begin{subarray}{c} u\in V\\ k(u,v) < K\end{subarray}} \!\!\! U_u$$
For $K,L\geq 1$, we call the quadruple $\mathcal A$ a \textit{$(K,L)$-approximation of $(X,d_g)$} if the following four conditions are satisfied. 
\begin{enumerate}[label=\normalfont{(A\arabic*)}]
	\item\label{a:1} Every vertex of $G$ has valence at most $K$.
	\item\label{a:2} $B(p_v,r_v)\subset U_v\subset B(p_v,Lr_v)$ for every $v\in V$.
	\item\label{a:3} Let $u,v\in V$. If $u\sim v$, then $U_u\cap U_v\neq \emptyset$ and $L^{-1}r_u\leq r_v\leq Lr_u$. Conversely, if $U_u\cap U_v\neq \emptyset$, then $k(u,v)<K$.
	\item\label{a:4} $N(U_v,{r_v/L})\subset \mathcal A\text{-}\st_K(v)$ for every $v\in V$.
\end{enumerate}
The $(K,L)$-approximation $\mathcal A$ of $X$ is called \textit{fine} if $U_v\neq X$ for all $v\in V$.\\

Given an $\varepsilon$-net $S \subset (X,d_g)$, we define a quadruple $\mathcal{A}_S$, as above, where $G$ is the graph with vertex set~$S$, with an edge connecting two vertices $s$ and $s'$ if $d(s,s') < 2\varepsilon$, $\mathfrak p$ is the inclusion of $S$ into $X$, $\mathfrak r$ is the constant $\varepsilon$, and $U_s := B(s,\varepsilon)$ for each $s \in S$.

\begin{lmm}\label{net_to_approx_lmm}
For each $n\in \N$, there is a $K \geq 1$ such that for each compact Riemannian $n$-manifold $(X,d_g)$, $\varepsilon > 0$ sufficiently small, and $\varepsilon$-net $S$, $\mathcal{A}_S$ is a $(K,1)$-approximation of $(X,d_g)$.
\end{lmm}

\begin{proof}
Conditions \ref{a:2} and \ref{a:3} with $L=1$ are clearly satisfied for any $K \geq 2$. By Lemma~\ref{local_control_lmm}, there are $c > 0$ and $M \geq 1$ such that for all $r \in (0,c)$ and $p \in X$ the $n$-volume of $B(p,r)$ is bounded below by $M^{-1}r^n$ and above by $Mr^n$.
Therefore, for every $\varepsilon > 0$ sufficiently small, at most $6^n M^2$ (resp. $8^n M^2$) disjoint balls of radius $\varepsilon/2$ can lie inside $B(s,3\varepsilon)$ (resp. $B(s, 4\varepsilon)$) for any $s \in S$.
By the $\varepsilon$-separation of $S$,
$$B(s',\varepsilon/2) \cap B(s'',\varepsilon/2) = \emptyset \quad \forall \, s', s'' \in S, \quad s' \neq s''.$$
Therefore, the cardinality of $S \cap B(s,2\varepsilon)$ is at most $6^n M^2$. Thus, condition~\ref{a:1} holds for any $K \geq 6^n M^2$.
By the $\varepsilon$-density condition,
\begin{equation}\label{net_to_approx_eqn}
N(U_s,r_u/L) = B(s,2\varepsilon) \subset \!\!\! \bigcup_{s' \in B(s,3\varepsilon)} \!\!\! U_{s'}.
\end{equation}
By the above, the cardinality of $S \cap B(s,3\varepsilon)$ is at most $8^n M^2$.
Since $B(s,2\varepsilon)$ is connected and $U_{s'} \cap U_{s''} \neq \emptyset$ whenever $s',s'' \in S$ are connected by an edge, any two vertices $s',s'' \in B(s,3\varepsilon)$ are connected in the graph through vertices lying in $B(s,3\varepsilon)$. Along with (\ref{net_to_approx_eqn}), this implies that condition~\ref{a:4} holds for any $K \geq 8^n M^2$.
\end{proof}

The statement of \cite[Theorem~2.8]{N25} uses $(2K+1)$-stars in $X$ determined by a fine $(K,L)$-approximation~$\mathcal{A}$ and the notion of \textit{$L$-bounded turning}. We do not define the latter notion, but note that it follows from Lemma~\ref{Hopf_Rinow_lmm} that compact Riemannian manifolds have $L$-bounded turning for every $L \geq 1$. Since $$\mathcal{A}_S\text{-}\st_{2K+1}(s) \subset B\big(s,2(2K+1)\varepsilon\big) \quad \forall \,s \in S,$$
we can take $R = 2(2K+1)$ for the purpose of applying \cite[Theorem~2.8]{N25} to obtain this proposition. 

\section{Smooth approximations}\label{main_sec}

This section contains the proofs of Theorems~\ref{quasi_thm} and \ref{Lipschitz_thm}. The idea for both proofs is to rescale the metric $d_g$ on $X$, then glue the rescaled metric with $d$ using Lemma~\ref{gluing_lmm}. For Theorem~\ref{Lipschitz_thm}, a global scaling works.

\begin{proof}[{\bf{\emph{Proof of Theorem~\ref{Lipschitz_thm}}}}]
By assumption, $d \leq L d_g$. Let $\varepsilon > 0$. By Lemma~\ref{net_lmm}, there is an $\varepsilon$-net $S$ in $(X,Ld_g)$. Let $\smoothmetric$ be the glued metric determined by $Ld_g$ and $d|_{S \times S}$. By Lemma~\ref{gluing_lmm}, $(X,\smoothmetric)$ is locally isometric to $(X,Ld_g)$ and $L$-bi-Lipschitz to $d_g$. The set $(S,d)$ is an $\varepsilon$-dense subset in both $(X,\smoothmetric)$ and $(X,d)$. Therefore, it is $\varepsilon$-close to both these spaces in the Gromov-Hausdorff distance and so these spaces are $2\varepsilon$-close to each other.
\end{proof}

For the rest of this section, we will adopt the following notation: $(X,d_g)$ is a connected Riemannian manifold and $d$ is another metric on $X$ such that $d_g$ is $\eta$-quasisymmetric to $d$, i.e.
\begin{equation}\label{particular_quasisymmetry_eqn}
d(p,q) \leq \eta \bigg(\frac{d_g(p,q)}{d_g(q,s)}\bigg)d(q,s) \quad \forall\, p,q \in X \quad \textnormal{with} \quad q \neq s
\end{equation}
The notations $B(p,r)$ and $\overline{B}(p,r)$ will only refer to balls taken with respect to $d_g$.\\

Quasisymmetries are more subtle than bi-Lipschitz homeomorphisms and a global scaling will not provide sufficient control to apply Proposition~\ref{main_technical_prp}. Therefore, we do a local rescaling. For $\varepsilon > 0$, define a function 
\begin{equation}\label{distortion_eqn}
\distortion: X \lra \R, \quad  \distortion(p) := \frac{\max_{q \in \overline{B}(p, \varepsilon)}d(p,q)}{\varepsilon}\,.
\end{equation}
We note that any $\varepsilon$-dense subset $S \subset (X,d_g)$ is $(\varepsilon \max_X \!\distortion)$-dense in $(X,d)$.
\begin{lmm}
The function $\distortion$ is continuous.
\end{lmm}

\begin{proof}
Let $(p_i)_{i \in \N}$ be a sequence of points in $X$ converging to $p_\infty$. We show that $\lim \distortion(p_i) = \distortion(p_\infty)$. For each $p_i$, let $q_i \in \overline{B}(p_i,\varepsilon)$ be a point realizing the maximum in (\ref{distortion_eqn}). By compactness, a subsequence of $(q_i)$ converges to a point $q_\infty$. It is clear that $q_\infty \in \overline{B}(p,\varepsilon)$ and that $d(p_\infty,q_\infty)/\varepsilon = \lim \distortion(p_i)$. Therefore, $\lim \distortion(p_i) \leq \distortion(p_\infty)$.\\

Conversely, let $s_\infty \in \overline{B}(p_\infty,\varepsilon)$ be a point realizing the maximum for $p_\infty$ in (\ref{distortion_eqn}). As $(X,d_g)$ is a length space, $$N(\overline{B}(p, r_1),r_2) = B(p,r_1+r_2) \quad \forall \, p \in X \quad \forall \,r_1,r_2 > 0.$$
As $p_i$ tend to $p_\infty$, it then follows that $s_\infty \in N(\overline{B}(p_i,\varepsilon),r_i)$ for $r_i$ tending to zero. Therefore, there exists a sequence $(s_i)_{i \in \N}$ converging to $s_\infty$ such that $s_i \in \overline{B}(p_i,\varepsilon)$. Thus, $\lim d(p_i,s_i)/\varepsilon = \distortion(p_\infty)$, and so $\lim \distortion(p_i) \geq \distortion(p_\infty)$.
\end{proof}

\begin{lmm}\label{sphere_bound_lmm}
For all $\varepsilon > 0$ and $p,q \in X$ such that $0 < d_g(p,q) \leq \varepsilon$,
$$\eta(\varepsilon/d_g(p,q))^{-1}\distortion(p)\varepsilon \leq d(p,q) \leq \distortion(p)\varepsilon.$$
\end{lmm}

\begin{proof}
Let $s \in \overline{B}(p,\varepsilon)$ be a point realizing the maximum for $p$ in (\ref{distortion_eqn}).
By (\ref{particular_quasisymmetry_eqn}),
$$\distortion(p) \varepsilon = d(p,s) \leq \eta(\varepsilon/d_g(p,q))d(p,q)$$
and the first inequality follows. The second inequality is immediate from (\ref{distortion_eqn}).
\end{proof}

In Lemmas~\ref{annular_bound_lmm}-\ref{main_lemma_on_metric} below, $\varepsilon > 0$ is assumed to be sufficiently small depending only on the Riemannian metric $g$ and a constant $R$.

\begin{lmm}\label{annular_bound_lmm}
For all $R\geq 1$, $\varepsilon > 0$ sufficiently small, and $p,q \in X$ with
$R^{-1}\varepsilon \leq d_g(p,q) \leq R\varepsilon$,
$$C^{-1}\distortion(p) \leq \distortion(q) \leq C \distortion(p),$$
where $C = \eta(1)\eta(R)^2$.
\end{lmm}

\begin{proof}
By Lemma~\ref{local_control_lmm}, there exist $p', q' \in X$ such that
$$d_g(p,p') = d_g(q,q') = \varepsilon.$$
Therefore, $d_g(p,p') \leq Rd_g(p,q)$ and $d_g(p,q) \leq Rd_g(q,q')$. Along with (\ref{particular_quasisymmetry_eqn}), these inequalities give
$$d(p,p') \leq \eta\bigg(\frac{d_g(p,p')}{d_g(p,q)}\bigg)d(p,q) \leq \eta(R)\eta\bigg(\frac{d_g(p,q)}{d_g(q,q')}\bigg)d(q,q') \leq \eta(R)^2 d(q,q').$$
Combining this with Lemma~\ref{sphere_bound_lmm} with $(p,q)$ replaced by $(p,p')$, we obtain
$$\distortion(p) \leq \eta(1)\frac{d(p,p')}{\varepsilon} \leq \eta(1)\eta(R)^2 \frac{d(q,q')}{\varepsilon} \leq \eta(1)\eta(R)^2\distortion(q).$$
This yields the first claimed inequality; the second follows by symmetry.
\end{proof}

Lemma~\ref{annular_bound_lmm} bounds $\distortion$ on a spherical shell around a point. It is more convenient to bound $\distortion$ on a ball. This is the content of Lemma~\ref{ball_bound_lmm} below. It follows from Lemma~\ref{annular_bound_lmm} by placing the ball inside a spherical shell around another point.

\begin{lmm}\label{ball_bound_lmm}
For all $R \geq 1$, $\varepsilon > 0$ sufficiently small, and $p, q \in X$ with
$d_g(p,q) < R\varepsilon$,
$$C^{-1}\distortion(p) \leq \distortion(q) \leq C\distortion(p),$$
where $C = \eta(1)^2\eta(8R)^4$.
\end{lmm}

\begin{proof}
By Corollary~\ref{throw_crl} with $r = R\varepsilon$, there exists a point $s \in X$ such that
$$8^{-1}R^{-1} \varepsilon \leq d_g(p,s) \leq 8R\varepsilon \quad \textnormal{and} \quad 8^{-1}R^{-1} \varepsilon \leq d_g(q,s) \leq 8R\varepsilon.$$
Applying Lemma~\ref{annular_bound_lmm} to $(p,s)$ and then to $(s,q)$ yields the claimed inequalities.
\end{proof}

Define the continuous Riemannian metric
$g_\varepsilon := \distortion^2 g$
and let $\continuousmetric$ be the induced metric on $X$. 

\begin{lmm}\label{main_lemma_on_metric}
For all $R \geq 1$, $\varepsilon > 0$ sufficiently small, and $p \in X$, $\continuousmetric$ is $C$-bi-Lipschitz to $\distortion(p)d_g$ on $B(p, R\varepsilon)$, where $C =\eta(1)^2\eta(16R)^4$.
\end{lmm}

\begin{proof}
For a rectifiable curve $\gamma$ in $X$, let $|\gamma|$ and $|\gamma|_\varepsilon$ be its lengths with respect to $g$ and $g_\varepsilon$, respectively. By Lemma~\ref{ball_bound_lmm}, for each $\varepsilon > 0$ sufficiently small, $p \in X$, and rectifiable curve $\gamma$ in $B(p,2R\varepsilon)$,
\begin{equation}\label{length_Lipschitz)inequality}
C^{-1} \distortion(p)|\gamma| \leq |\gamma|_\varepsilon \leq C \distortion(p)|\gamma|.
\end{equation}
By Lemma~\ref{local_control_lmm} with $r = R\varepsilon$, for each $\varepsilon > 0$ sufficiently small and $p\in X$, $B(p,R\varepsilon)$ is convex with respect to $g$. Thus, the $d_g$-distance between points in each ball is realized by a rectifiable curve lying inside the ball.
Let $q,s \in B(p,R\varepsilon)$. The distance $\continuousmetric(q,s)$ equals the minimum $\continuousmetric$-length of a rectifiable curve $\gamma$ connecting $q$ to $s$ and thus
$$\continuousmetric(q,s) \leq C\distortion(p)d_g(q,s)$$
by the $g$-convexity of $B(p,R\varepsilon)$ and (\ref{length_Lipschitz)inequality}).
If this curve lies inside $B(p,2R\varepsilon)$, then it follows from (\ref{length_Lipschitz)inequality})
that
\begin{equation}\label{length_lip_2}
C^{-1}\distortion(p)d_g(q,s) \leq \continuousmetric(q,s).
\end{equation}
If $\gamma$ does not lie inside $B(p,2R\varepsilon)$, then it contains segments connecting the boundary of $B(p,2R\varepsilon)$ to $q$ and to $s$. By (\ref{length_Lipschitz)inequality}) and the fact that $q,s \in B(p,R\varepsilon)$, each of these segments has $\continuousmetric$-length at least $C^{-1}\distortion(p)R\varepsilon$, so $\continuousmetric(q,s) \geq 2C^{-1}\distortion(p)R\varepsilon$. Lastly, as the diameter of $B(p,R\varepsilon)$ with respect to $g$ is at most $2R\varepsilon$, then $d_g(q,s) \leq 2R\varepsilon$, so (\ref{length_lip_2}) still holds and the result follows.
\end{proof}

\begin{proof}[{\bf{\emph{Proof of Theorem~\ref{quasi_thm}}}}]
Let $R\geq1$ be sufficiently large (depending only on the dimension $n$ of $X$) and $\varepsilon > 0$, as in Proposition~\ref{main_technical_prp} in both cases. Define $C := \eta(1)^2\eta(16R)^4$. Let $S_\varepsilon \subset (X,d_g)$ be an $\varepsilon/2$-net.\\

By \cite[Theorem~4.45]{F99}, every continuous function on a smooth manifold is the uniform limit of a sequence of smooth functions. Let $\wt{\distortion}: X \lra (0,\infty)$ be a smooth function such that
\begin{equation}\label{smoothing_of_dist_bound}
2^{-1}\distortion(p) \leq \wt{\distortion}(p) \leq 2\distortion(p) \quad \forall \, p \in X.
\end{equation}
Let $\smoothmetric$ be the metric on $X$ induced by the Riemannian metric $h_\varepsilon := (4C\wt{\distortion})^2 g.$ We show below first that $d \leq \smoothmetric$ on $S_\varepsilon \times S_\varepsilon$ for all $\varepsilon > 0$ sufficiently small. Thus, the glued metric $\wt{\smoothmetric}$ of Lemma~\ref{gluing_lmm} determined by $\smoothmetric$ and $d_{S_\varepsilon \times S_\varepsilon}$ is well-defined and is locally isometric to $\smoothmetric$. We then use Proposition~\ref{main_technical_prp} to show that the metrics $\wt{\smoothmetric}$ are uniformly quasisymmetric to $d_g$. As the last step of the proof, we note that these metrics converge to $d$ in the Gromov-Hausdorff sense as $\varepsilon$ tends to zero.\\

\textit{Quasisymmetry condition.}
We now verify that the assumptions of Proposition~\ref{main_technical_prp} with $\varepsilon$ replaced by $\varepsilon/2$ are satisfied by $d_g$, $\smoothmetric$, and $d$. By assumption, $d_g$ is $\eta_1$-quasisymmetric to $d$ for any homeomorphism $\eta_1: {[}0,\infty{)} \lra {[}0,\infty{)}$ such that $\eta_1 \geq \eta$.
Let $s \in S_\varepsilon$. By Lemma~\ref{main_lemma_on_metric}, $\distortion(s)d_g$ is $C$-bi-Lipschitz to $\continuousmetric$ on $B(s,R\varepsilon)$. Thus, $\distortion(s)d_g$ is $8C^2$-bi-Lipschitz to $\smoothmetric$ on $B(s,R\varepsilon)$ by (\ref{smoothing_of_dist_bound}). Therefore, by Lemma~\ref{bi-Lip_are_QS_lmm}, $\distortion(s)d_g$ and $d_g$ are $\eta_1$-quasisymmetric to $\smoothmetric$ on $B(s,R\varepsilon)$ for any homeomorphism $\eta_1: {[}0,\infty{)} \lra {[}0,\infty{)}$ such that $\eta_1(t) \geq (8C^2)^2t$.\\

Let $s,s' \in X$ satisfy $\varepsilon/2 \leq d_g(s,s') < \varepsilon$. Thus, $\eta(\varepsilon/d_g(s,s')) \leq \eta(2)$.
By Lemma~\ref{sphere_bound_lmm} combined with the bounds on $d_g(s,s')$,
$$\eta(2)^{-1}\distortion(s)d_g(s,s') \leq \eta(2)^{-1} \distortion(s)\varepsilon \leq d(s,s') \leq \distortion(s)\varepsilon \leq 2\distortion(s)d_g(s,s').$$
Along with Lemma~\ref{main_lemma_on_metric}, this gives
$$C^{-1}\eta(2)^{-1}\continuousmetric(s,s') \leq d(s,s') \leq 2C\continuousmetric(s,s').$$
As $\smoothmetric/4C$ is $2$-bi-Lipschitz to $\continuousmetric$, $4^{-1}\smoothmetric(s,s') \leq 2C\continuousmetric(s,s') \leq \smoothmetric(s,s')$. Therefore,
\begin{equation}\label{short_on_adj_points_inequality}
\big(8C^2 \eta(2)\big)^{-1}\smoothmetric(s,s') \leq d(s,s') \leq \smoothmetric(s,s').
\end{equation}
By the first inequality in (\ref{short_on_adj_points_inequality}), (\ref{main_technical_ineq}) holds with $\varepsilon$ replaced by $\varepsilon/2$ for any $L \geq \max \{1, 8C^2 \eta(2)\}$.\\

Suppose now $p,q \in X$ satisfy $d_g(p,q) \geq \varepsilon/2$. Let $\gamma$ be a length-minimizing geodesic with respect to $h_\varepsilon$ connecting $p$ to $q$, as provided by Lemma~\ref{Hopf_Rinow_lmm}. Let $s_0,s_1,\dots s_l$ be a string of points on $\gamma$ such that $s_0 = p$, $s_l = q$, and $\varepsilon/2 \leq d_g(s_i,s_{i+1}) < \varepsilon$ for $0 \leq i \leq l-1$; such a string can be generated by iteratively taking midpoints\footnote{A continuity argument shows that every continuous path contains a point equidistant from its endpoints.}. Then, by the triangle inequality, the second inequality in (\ref{short_on_adj_points_inequality}), and the fact that $\gamma$ is length-minimizing,
$$d(p,q) \leq \sum_0^{l-1}d(s_i,s_{i+1}) \leq \sum_0^{l-1}\smoothmetric(s_i,s_{i+1}) = \smoothmetric(p,q).$$
Therefore, $d \leq \smoothmetric$ on $S_\varepsilon \times S_\varepsilon$, as claimed. By Proposition~\ref{main_technical_prp}, there exists a homeomorphism

$\eta_2: {[}0,\infty{)} \lra {[}0,\infty{)},$ depending only on $n$ and $\eta$, such that $d_g$ is $\eta_2$-quasisymmetric to $\wt{\smoothmetric}$.\\

\textit{Gromov-Hausdorff convergence.}
Let $\mu_\varepsilon\! :=\! \max \{C,1\}(\max_X\!\distortion).$ The set $S_\varepsilon$ is an $\mu_\varepsilon \varepsilon$-dense subset of $(X,d)$ and $(X,\wt{\smoothmetric})$. Therefore, it is $\mu_\varepsilon \varepsilon$-close in the Gromov-Hausdorff distance to each of these spaces, so they are $2\mu_\varepsilon \varepsilon$-close to each other. It therefore suffices to show that $2\mu_\varepsilon \varepsilon$ tends to zero as $\varepsilon$ tends to zero. By Lemma~\ref{Holder_lmm}, there are $C' \geq 1$ and $0 < \alpha \leq 1$ such that $\varepsilon\distortion(p) \leq C'\varepsilon^{\alpha}$ for every point $p \in X$. The claim follows.
\end{proof}

\begin{rmk}\label{remove_locally_rmk}
If the metric space $(X,d)$ in Theorem~\ref{quasi_thm} is a length space, then the Riemannian manifolds $(X,\smoothmetric)$ in the proof of Theorem~\ref{quasi_thm} are uniformly quasisymmetric to $(X,d_g)$ and tend to $(X,d)$ in the Gromov-Hausdorff sense. This follows from the fact that the metrics $\smoothmetric$ are then uniformly $(8C^2\eta(2))$-bi-Lipschitz to $\wt{\smoothmetric}$, which we now prove. By Lemma~\ref{gluing_lmm}, it suffices to show that $(8C^2\eta(2))d \geq \smoothmetric$ on $S_\varepsilon \times S_\varepsilon$. This follows from essentially the same argument as used in the proof of Theorem~\ref{quasi_thm} to show that $d \leq \smoothmetric$ on $S_\varepsilon \times S_\varepsilon$. Specifically, let $p,q \in X$ 
satisfy $d_g(p,q) \geq \varepsilon/2$. As $d$ is a length metric, for every $\varepsilon'>0$, there exists a string of points $s_0,s_1,\dots,s_l \in X$ such that
$s_0 = p$, $s_l = q$,
$$\varepsilon/2 \leq d_g(s_i,s_{i+1}) < \varepsilon ~~\textnormal{for}~~ 0\leq i \leq l-1 \quad \textnormal{and} \quad
d(p,q) + \varepsilon' \geq \sum_0^{l-1} d(s_i,s_{i+1}).$$
By the first inequality in (\ref{short_on_adj_points_inequality}), the triangle inequality, and the fact that $\varepsilon'$ is arbitrary, it follows that $(8C^2\eta(2))d(p,q) \geq \smoothmetric(p,q)$, as claimed.
\end{rmk}

\begin{rmk}\label{lack_of_compactness_rmk}
The compactness of $X$ is needed at two points in the proof of Theorem~\ref{quasi_thm}. It is used in Lemma~\ref{Holder_lmm}, which is applied in the proof to show the Gromov-Hausdorff convergence of the constructed sequence. If $X$ were non-compact, one could instead consider pointed Gromov-Hausdorff convergence. Compactness is also used in Lemma~\ref{local_control_lmm}. If $X$ were non-compact, this could be circumvented by assuming additionally that the Riemannian manifold $(X,d_g)$ has uniformly bounded geometry. This is a reasonable condition in practice.
\end{rmk}

\vspace{.3in}

{\it Department of Mathematics, Stony Brook University, Stony Brook, NY 11794\\
spencer.cattalani@stonybrook.edu}

\end{document}